\documentclass{amsart}

\usepackage{amssymb}
\usepackage{amsmath}
\usepackage{amsthm}
\usepackage[dvips]{graphicx}
\usepackage{color}

\definecolor{cyan}{cmyk}{1,0,0,0}
\definecolor{lightcyan}{cmyk}{0.5,0,0,0}
\definecolor{pastelcyan}{cmyk}{0.25,0,0,0}
\definecolor{magenta}{cmyk}{0,1,0,0}
\definecolor{yellow}{cmyk}{0,0,1,0}
\definecolor{lightyellow}{cmyk}{0,0,0.5,0}
\definecolor{pastelyellow}{cmyk}{0,0,0.25,0}
\definecolor{black}{cmyk}{0,0,0,1}
\definecolor{darkgray}{cmyk}{0,0,0,0.75}
\definecolor{gray}{cmyk}{0,0,0,0.5}
\definecolor{lightgray}{cmyk}{0,0,0,0.25}
\definecolor{white}{cmyk}{0,0,0,0}
\definecolor{red}{cmyk}{0,1,1,0}
\definecolor{orange}{cmyk}{0,0.5,1,0}
\definecolor{scarlet}{cmyk}{0,1,0.5,0}
\definecolor{brown}{cmyk}{0.5,0.75,1,0}
\definecolor{camel}{cmyk}{0.25,0.375,0.5,0}
\definecolor{cream}{cmyk}{0,0.2,0.3,0}
\definecolor{green}{cmyk}{1,0,1,0}
\definecolor{lightgreen}{cmyk}{0.5,0,0.5,0}
\definecolor{pastelgreen}{cmyk}{0.25,0,0.25,0}
\definecolor{mossgreen}{cmyk}{0.64,0.4,1,0}
\definecolor{yellowgreen}{cmyk}{0.5,0,1,0}
\definecolor{skyblue}{cmyk}{0.4,0.16,0,0}
\definecolor{royal}{cmyk}{1.0,0.5,0,0}
\definecolor{navyblue}{cmyk}{0.9,0.75,0.5,0}
\definecolor{lightnavy}{cmyk}{0.4,0.3,0.2,0}
\definecolor{blue}{cmyk}{1,1,0,0}
\definecolor{lightblue}{cmyk}{0.5,0.5,0,0}
\definecolor{lavender}{cmyk}{0.25,0.25,0,0}
\definecolor{violet}{cmyk}{0.75,1,0.25,0}
\definecolor{purple}{cmyk}{0.5,1,0.5,0}
\definecolor{pink}{cmyk}{0,0.5,0,0}
\definecolor{pastelpink}{cmyk}{0,0.25,0,0}
 \newtheorem{theo}{Theorem}[section]
 \newtheorem{lemm}[theo]{Lemma}

\newtheorem{defi}[theo]{Definition}

\newtheorem{prop}[theo]{Proposition}

\newcommand\CC{{\mathbb C}}
\newcommand\NN{{\mathbb N}}
\newcommand\RR{{{\mathbb R}}}

\def\SS {\mathbb{S}}

\def\la{\langle}
\def\ra{\rangle}

\let\ve=\varepsilon

\newcommand\cA{{\mathcal A}}

\newcommand\cC{{\mathcal C}}

\newcommand\cF{{\mathcal F}}

\newcommand\cS{{\mathcal S}}

\newcommand\cK{{\mathcal K}}

\newcommand\pa{\partial}

\def\vb{\mathbf b}
\def\ve{\mathbf e}

\date{4-August-2012}
\begin{document}
\title[Villani conjecture]
{Villani  conjecture 
on smoothing effect \\
of the homogeneous Boltzmann equation with \\ measure initial datum 
}
\author{Y. Morimoto }
\address{Y. Morimoto, Graduate School of Human and Environmental Studies,
Kyoto University,
Kyoto, 606-8501, Japan} \email{morimoto@math.h.kyoto-u.ac.jp}
\author{T. Yang}
\address{T. Yang, Department of mathematics, City University of Hong Kong,
Hong Kong, P. R. China
} \email{matyang@cityu.edu.hk}



\date{}

\begin{abstract}
We justify the Villani conjecture on the smoothing effect for measure value solutions to the space homogeneous Boltzmann equation of Maxwellian type cross sections. This is the first rigorous proof of the smoothing effect
for any measure value initial data except the single Dirac mass, which gives the optimal description on the 
regularity of  solutions for positive time, caused by the singularity in the cross section. 
The main new ingredient in the proof is the introduction of a time degenerate coercivity estimate by using  the microlocal analysis.
\end{abstract}
\maketitle

\section{Introduction}\label{s1}

We consider the
spatially homogeneous Boltzmann equation 
\begin{equation}\label{bol}
\partial_t f(t,v) 
=Q(f, f)(t,v),
\end{equation}
where $f(t,v)$ is the density distribution of particles with 
velocity $v \in \RR^3$ at time $t$.
The right hand side of (\ref{bol}) is given by the
Boltzmann bilinear collision operator
\[
Q(g, f)(v)=\int_{\RR3}\int_{\mathbb S^{2}}B\left({v-v_*},\sigma
\right)
 \left\{g(v'_*) f(v')-g(v_*)f(v)\right\}d\sigma dv_*\,,
\]
where for $\sigma \in \SS^2$
$$
v'=\frac{v+v_*}{2}+\frac{|v-v_*|}{2}\sigma,\,\,\, v'_*
=\frac{v+v_*}{2}-\frac{|v-v_*|}{2}\sigma,\,
$$
which follow from the conservation of momentum and energy,
\[ v' + v_*' = v+ v_*, \enskip |v'|^2 + |v_*'|^2 = |v|^2 + |v_*|^2.
\]
The equation \eqref{bol} is supplemented with a non-negative initial datum
\begin{equation}\label{initial}
f(0,v) = f_0(v),
\end{equation}
which is the density of probability distribution
(more generally a probability measure).

The non-negative cross section
 $B(z, \sigma)$ depends only on $|z|$ and the scalar product
$\frac{z}{|z|}\,\cdot\, \sigma$. 
For physical models, 
it usually takes the form
\begin{equation*}
B(|v-v_*|, \cos \theta)=\Phi (|v-v_*|) b(\cos \theta),\,\,\,\,\,
\cos \theta=\frac{v-v_*}{|v-v_*|} \, \cdot\,\sigma\, , \,\,\,
0\leq\theta\leq\frac{\pi}{2},
\end{equation*}
where
\begin{align}\label{1.2-0}
&\Phi(|z|)=\Phi_\gamma(|z|)= |z|^{\gamma}, \enskip \mbox{for some $\gamma>-3$}, \\
& b(\cos \theta)\theta^{2+2s}\ \rightarrow K\ \
 \mbox{when} \ \ \theta\rightarrow 0+,  \enskip 
\mbox{for $0<s<1$ and $K>0$. }\label{1.2}
\end{align}
ln fact, if the inter-particle potential satisfies the inverse power law 
$U(\rho)=\rho^{-(q-1)}$, $q >2$, where $\rho$ denotes the distance between two interacting particles, then
$s$ and $\gamma$ are given by
\[s = 1/(q-1) \,  <1\, , \enskip \enskip \gamma = 1 -4s = 1-4/(q-1) \,\, >-3\,.
\]
For this physical model, we have $\gamma =0$ if $s=1/4$, which is called the 
Maxwellian molecule. 
Inspired by  this case, we consider the Maxwellian molecule type cross section
when  
\[
\gamma = 0,\enskip 0<s<1\,.
\]


The angle $\theta$ is the deviation angle, i.e., the angle between 
pre- and post- collisional velocities. The range of $\theta$ is  in an interval $[0,\pi]$, but as in \cite{villani2} it is customary  to 
restrict it to $[0,\pi/2]$, by replacing $b(\cos\theta)$
 by its ``symmetrized'' version  
\[
[ b(\cos \theta)+b(\cos (\pi-\theta))]{\bf 1}_{0\le \theta\le \pi/2},
\]
which is possible due to the invariance of the product $f(v')f(v'_*)$ in the collision operator $Q(f,f)$ 
under the change of variables $\sigma \to -\sigma$.
 It should be noted that $b(\cos \theta)$ has the integrable singularity, that is,
\[
\int_{\SS^2} b\Big(\frac{v-v_*}{|v-v_*|}\cdot \sigma\Big) d\sigma = 2\pi \int_0^{\pi/2} b(\cos \theta) \sin \theta d \theta = \infty\,.
\]
The case where $0<s<1/2$, that is, $\int_0^{\pi/2} \theta b(\cos \theta) \sin \theta d \theta < \infty$ is called the mild singularity, and 
another case $1/2 \le s <1$ is called the strong singularity.
This kind of  singularity leads to the gain of regularity in the solution.

The study on the homogeneous Boltzmann equation has a very long history, cf.
\cite{carleman, arkeryd} and the references in recent work \cite{lu-mouhot}. 
In particular,
the smoothing effect of (weak) solutions to the Cauchy problem for the non cutoff
homogeneous Boltzmann equation has been studied by
many authors in \cite{desv-wen1,A-Sa,A-Sa2,MUXY-DCDS,HMUY,AMUXY-KJM,H-C}, including 
Gevrey smoothing effect in \cite{MU}.  
However, the problem for  measure initial data has been studied only in 
\cite{morimoto-12}, when it consists of a sum  of four Dirac masses.

On the other hand, Villani conjecture\cite{villani3} is to show that 
the smoothing effect for weak measurable solutions holds for any measure 
initial data except a single Dirac mass. The purpose of this paper
is to justify this conjecture, which is optimal in the sense that
a single Dirac mass is a stationary
solution of the Boltzmann equation.
 
Let us now introduce some notations for function spaces and recall some
related works on the existence and uniqueness. 
For every $0\le \alpha < \infty$, we denote by $P_\alpha(\RR^d)$ the class of all probability measure
$F$ on $\RR^d$, $d \ge 1$, such that
\[
\int_{\RR^d} |v|^\alpha dF(v) < \infty\,.
\]
Concerning the Cauchy problem 
for the homogeneous Boltzmann equation of the Maxwellian molecule type cross section,
Tanaka \cite{tanaka78} in 1978 proved  the existence and the uniqueness of the solution in the space $P_2(\RR^d)$ by using probability theory. 
The proof of this result was simplified and generalized in  \cite{PT, toscani-villani}.

The existence of solution with bounded energy was extended 
in  \cite{Cannone-Karch} to the initial datum as
a probability measure with infinite energy. Precisely, 
following \cite{Cannone-Karch}, introduce 

\begin{defi}\label{cha-def}
A function $\psi : \RR^3 \rightarrow \CC$ is called a characteristic function
if there is a probability measure $\Psi$ ( i.e., a positive Borel measure with $\int_{\RR^3} d \Psi(v) =1$)
such that the identity $\psi(\xi) = \int_{\RR^3} e^{-iv\cdot \xi}d \Psi(v) $
holds.
We denote the set of all characteristic functions by $\cK$.
\end{defi}
Inspired by \cite{toscani-villani}, 
a subspace  $\cK^\alpha$ for $\alpha \ge 0$ was defined in \cite{Cannone-Karch}
as follows:
\begin{align}\label{K-al}
\cK^\alpha =\{ \varphi \in \cK\,;\, \|\varphi - 1\|_{\alpha} < \infty\}\,,
\end{align}
where 
\begin{align}\label{dis-norm}
\|\varphi - 1\|_{\alpha} = \sup_{\xi \in \RR^3} \frac{|\varphi(\xi) -1|}{|\xi|^\alpha}.
\end{align}
The space $\cK^\alpha $ endowed with the distance 
\begin{align}\label{distance}
\|\varphi - \tilde \varphi\|_{\alpha} = \sup_{\xi \in \RR^3} \frac{|\varphi(\xi) -
\tilde \varphi(\xi) |}{|\xi|^\alpha}
\end{align}
is a complete metric space (see Proposition 3.10 of \cite{Cannone-Karch}).
It follows that $\cK^\alpha =\{1\}$ for all $\alpha >2$ and the embeddings
(Lemma 3.12 of \cite{Cannone-Karch}) hold, that is, 
\[
\{1\} \subset \cK^\alpha \subset \cK^\beta \subset \cK^0 =\cK \enskip \enskip
\mbox{for all $2\ge \alpha \ge \beta \ge 0$}\,\,.
\]
The defintion of the space $\cK^\alpha$ is natural because we have the following lemma
(Lemma 3.15 of \cite{Cannone-Karch}).
\begin{lemm}\label{measure-positive}
Let $\Psi$ be a probability measure on $\RR^3$ such that
\begin{align}\label{moment-assumption}
\begin{array}{l} 
\displaystyle \exists \alpha \in (0, 2] ;\, \int |v|^\alpha d\Psi(v) <\infty \,, \\
\displaystyle \mbox{and  moreover,} \enskip
\int v_j d \Psi(v) =0\,, j=1,2,3, \enskip \mbox{when $\alpha >1$}.
\end{array}
\end{align}
Then the Fourier transform of $\Psi$, that is, $\psi(\xi) = \int e^{-iv\cdot \xi} d \Psi(v)$ belongs to 
$\cK^\alpha$.
\end{lemm}
The inverse of the lemma does not hold, in fact, the space $\cK^\alpha$ is bigger than the set of the Fourier
transform of $P_\alpha$ 
 (Remark 3.16 of \cite{Cannone-Karch}).
So we introduce 
$\tilde P_\alpha = \cF^{-1}(\cK^\alpha)$ endowed also with the distance \eqref{distance}.
The existence and the uniqueness of the solution in the space $\tilde P_\alpha$
was proved in \cite{Cannone-Karch} for the mild singularity,
and has been recently improved in \cite{morimoto-12} for the strong singularity.
Namely, if the cross section $b(\cos \theta)$ satisfies 
\eqref{1.2-0} with $0<s<1$ and if $2s < \alpha \le 2$, then there exists
a unique solution to the Cauchy problem  \eqref{bol}-\eqref{initial} 
in the space $C([0,\infty), \tilde P_\alpha)$ for 
any initial datum in $\tilde P_\alpha$ (see Theorem \ref{remark-ck}
in the Appendix).

We are now ready to state the main results of this paper.

\begin{theo}\label{main-thm} 
Let $b(\cos \theta)$ satisfy \eqref{1.2} with $0<s<1$
and let $\alpha \in (2s, 2]$.
If $F_0 \in \tilde P_\alpha(\RR^3)$ is not a single Dirac mass and  $f(t,v)$ is a unique solution 
in $C([0,\infty), \tilde P_\alpha)$ to the Cauchy problem \eqref{bol}-\eqref{initial},
then there exists a $T >0$ such that $f(t, \cdot ) \in H^\infty(\RR^3)$ for any $0 <t \le T$. 
Moreover, $T=\infty$ when   $F_0 \in P_2(\RR^3)$ for $0<s<1$,
and when $F_0 \in P_1(\RR^3)$ for $0<s<\frac 12$.
\end{theo}



\begin{lemm}\label{time-deg-thm}
Let $F_0 \in \tilde P_\alpha(\RR^3)$ and $f(t,v) \in 
C([0,\infty), \tilde P_\alpha)$ be the same as in Theorem \ref{main-thm}.
If $\psi(t,\xi)$ and $\psi_0(\xi)$ are Fourier transforms of $f(t,v)$ and $F_0$, respectively, 
then there exist $T >0$ and $C >0$, such that for $t \in [0, T]$ we have
\begin{align}\label{time-deg-coer}
t \int_{\RR^3} \langle \xi \rangle^{2s}|h(\xi)|^2 d \xi &\le  C \Big(\int_{\RR^3} \Big(\int_{\SS^2}b\Big(\frac{\xi}{|\xi|}\cdot \sigma\Big) (1 - |\psi(t,\xi^-)| ) d\sigma \Big)
|h(\xi)|^2 d \xi  \\&
\qquad \qquad \qquad +   \int_{\RR^3} |h(\xi)|^2 d \xi \Big ), \enskip \, \enskip \mbox{for} \enskip \forall h \in L^2_{s}, \notag
\end{align}
where $\xi^- = (\xi - |\xi|\sigma)/2$.
\end{lemm}

With Lemma \ref{time-deg-thm}, the proof of Theorem \ref{main-thm} can be given 
as follows.

\begin{proof}[Proof of Theorem \ref{main-thm}] 
It follows from the Bobylev formula that the Cauchy problem \eqref{bol}-\eqref{initial} is reduced to 
\begin{equation}\label{c-p-fourier}
\left \{ 
\begin{array}{l}\displaystyle \partial_t \psi(t,\xi)
=\int_{\SS^2}b\left(\frac{\xi \cdot \sigma}{|\xi|}\right) \Big( \psi(t,\xi^+)\psi(t, \xi^-) - \psi(t, \xi)
\psi(t,0)\Big) d\sigma, \\\\
\displaystyle \psi(0,\xi)=\psi_0(\xi), \enskip \mbox{where} \enskip
\displaystyle \xi^\pm = \frac{\xi}{2} \pm \frac{|\xi|}{2} \sigma\,.
\end{array}
\right.
\end{equation}
By Theorem \ref{remark-ck}, $\psi(t,\xi) \in 
C([0,\infty),\cK^\alpha)$. 
Define a time dependent weight function 
\[
M_\delta(t,\xi) 
= \la \xi \ra^{Nt^2-4} \la \delta \xi\ra^{-2N_0}\,, \enskip \enskip \la \xi \ra^2 = 1+|\xi|^2\,,
\]
where  $N_0 = NT^2/2 +2$, $N \in \NN$ and $\delta >0$.  We multiply  the
first equation of \eqref{c-p-fourier} by $M_\delta(t,\xi)^2 \overline{\psi(t,\xi)}$ 
and integrate with respect to $\xi$ over $\RR^3$.  Denote $\psi^{\pm} = \psi(t,\xi^{\pm})$ and $M^+ =  M_\delta(t,\xi^+)$ to simplify the notation and
note that
\begin{align*}
-2 \mbox{Re} \, \, \Big\{(\psi^+\psi^- - \psi) M^2 \overline{\psi}\Big\}&=\Big( |M\psi|^2 +|M^+\psi^+|^2 - 
2\mbox{Re} \, \Big\{\psi^- (M^+\psi^+)\overline{M\psi}\Big\}\Big)\\
+ \Big(|M\psi|^2 &-|M^+\psi^+|^2\Big) + 2\mbox{Re} \, \Big\{\psi^- \big((M-M^+)\psi^+\big)\overline{M\psi}\Big\}\\
& = J_1 + J_2 + J_3\,.
\end{align*}
Using the Cauchy-Schwarz inequality for the third term of $J_1$, we have 
\begin{align*}
J_1 \ge (1 -|\psi^-|) \Big(|M\psi|^2 +|M^+\psi^+|^2\Big)  \ge  (1 -|\psi^-|) |M\psi|^2\,.
\end{align*}
Therefore, by means of \eqref{time-deg-coer} we get
\begin{align}\label{coer}
\int_{\RR^3\times \SS^2} b\Big(\frac{\xi}{|\xi|}\cdot \sigma\Big)J_1 d\sigma d\xi 
+ \int_{\RR^3} |M\psi|^2  d\xi 
 \gtrsim t \int_{\RR^3}\la \xi\ra^{2s}|M\psi|^2 d\xi,
\end{align}
where $A \gtrsim B$ means that there exists a constant $C_0 >0$
such that $A \ge C_0 B$.
If we use the change of variable $\xi \rightarrow  \xi^+$ for the term $M^+\psi^+$ in $J_2$, by  the cancellation
lemma (Lemma 1 of \cite{ADVW}), we have
\begin{align*}
&\left|\int_{\RR^3\times \SS^2} b\Big(\frac{\xi}{|\xi|}\cdot \sigma\Big)J_2 d\sigma d\xi \right| \\
&\qquad  =2\pi  \left|\int_{\RR^3} |M\psi|^2 \Big( \int_0^{\pi/2} b(\cos \theta )\sin \theta \Big( 1- \frac{1}{\cos^3 (\theta/2)}\Big) d\theta\Big) d\xi\right|\\
 &\qquad \lesssim  \int_{\RR^3} |M\psi|^2 d\xi\,,
\end{align*}
where $A \lesssim B$ means that there exists a constant $C_0 >0$
such that $A \le C_0 B$.
Since $|M-M^+| \lesssim \sin^2(\theta/2) M^+$ (see (3.4) of \cite{MUXY-DCDS}), by the Cauchy-Schwarz inequality
we also have the same upper bound estimate for $J_3$
by using again the change of variable $\xi \rightarrow  \xi^+$ for the term including $M^+\psi^+$.
Since 
\begin{align*}
2 \mbox{Re}\,\Big(\frac{\pa \psi}{\pa t} M^2 \overline{\psi} \Big) =
\frac{\pa |M\psi|^2 }{\pa t}  - 4N t \log \la \xi \ra  |M\psi|^2\,,
\end{align*}
and 
 $|\xi|^{2s}/ \log \la \xi \ra  \rightarrow \infty$ as $|\xi| \rightarrow \infty$, we have  
\begin{align*}
\frac{d} {dt} \int_{\RR^3} |M_{\delta}(t,\xi)\psi(t,\xi)|^2 d\xi  \lesssim \int_{\RR^3} |M_{\delta}(t,\xi)\psi(t,\xi)|^2 d\xi  \,,
\end{align*}
which gives for $t \in (0,T]$
\[
\int_{\RR^3} |   \la \xi \ra^{Nt^2-4} \Big(1+\delta |\xi|^2\Big)^{-N_0} \psi(t,\xi)|^2 d\xi 
\lesssim \int_{\RR^3} |   \la \xi \ra^{-4} \psi_0(\xi)|^2 d\xi \,.
\]
Letting $\delta \rightarrow 0$, we obtain the first part of
Theorem \ref{main-thm} because we can take an arbitrarily large $N$. 

We now turn to the second part of the theorem when $F_0 \in P_2(\RR^3)$. 
{We notice  that the energy of solution is uniformly bounded by 
that of the initial datum (see Proposition \ref{energy-con}
in the Appendix), so that
we have $\int |v|^2 f(T,v)dv \leq  \int |v|^2 dF_0(v)$ for a $T>0$
given in Lemma \ref{time-deg-thm}.} 	In view of 
$f(T,v) \in L^\infty(\RR^3)$ we obtain 
\[
\|f(T) \|_{L \log L}: =\int f(T,v) \log( 1+ f(T,v))dv < \infty,
\]
so that $f(T) \in L^1_2 \cap L \log L$. It follows from Theorem 1 in \cite{villani} that
\begin{equation}\label{uniform}
\sup_{t \ge T} \Big(\|f(t)\|_{L^1_2} + \|f(t)\|_{L \log L} \Big) < \infty,
\end{equation}
which shows that there exists a  $\kappa >0$ independent of $t \ge T$ 
such that
\[
1 - |\psi(t, \xi)| \ge \kappa \min (1, |\xi|^2),
\]
by means of Lemma 3 in \cite{ADVW}. Therefore, for $|\xi| \ge R$ for
some $R >0$ suitably large, we have
\begin{align*}
\int_{\SS^2}b\Big(\frac{\xi}{|\xi|}\cdot \sigma\Big) (1 - |\psi(t,\xi^-)| ) d\sigma
&\ge 2 \pi \kappa \int_0^{|\xi|^{-1}} b(\cos \theta) 
|\xi^-|^2 \sin \theta d\theta\\
&\gtrsim |\xi|^2 \int_0^{|\xi|^{-1} }\theta^{1-2s}d\theta \gtrsim |\xi|^{2s},
\end{align*}
which gives the standard  coercivity estimate instead of \eqref{time-deg-coer}.
Hence this leads us to $f(t,v) \in H^\infty(\RR^3)$ for $\forall t >T$
by the same argument used in \cite{MUXY-DCDS}.

It remains to show the last statement in the theorem for the case
when $F_0\in P_1(\RR^3)$ with $0<s<\frac 12$. In fact, it follows from the almost same argument in the above proof of the second part, if one uses Proposition
\ref{entropy-finite} in the Appendix with $t=0$  replaced by a small $t=T$ given in  Lemma \ref{time-deg-thm}.
\end{proof}
\bigskip
The rest of the paper will be organized as follows. In the next section, we will
prove Lemma \ref{time-deg-thm} about the degenerate coercivity estimate which is the key estimate to show the smoothing effect. And in the Appendix, we will recall the existence and uniquess result obtained in
 \cite{Cannone-Karch, morimoto-12} and show the continuity of the time derivative of the solution which is needed in Section \ref{s2}. 
It will be also shown in the Appendix that the energy of the solution for 
 the initial datum $F_0 \in P_2(\RR^3)$ is bounded.

\section{Degenerate coercivity estimate}\label{s2}
\setcounter{equation}{0}
To obtain the coercivity estimate for measure value function which is not
concentrated at a single point, we will consider two cases, that is, the case
when the measure is concentrated on a straight line and otherwise. Unlike the
standard coercivity estimate obtained in the previous works, the key observation
is that the coercivity estimate is degenerate in the time variable as shown
in Lemma \ref{time-deg-thm}. That is, one can not expect to have a gain
of regularity of order $2s$ uniformly up to initial time. For this, we need to consider
the time derivative of $\psi(t,\xi^-)$ in the case when $\xi$ is parallel to the
straight line of the concentration of the measure. For clear presentation, the coercivity is estimated in the following two subsections.

\subsection{Initial measure not concentrated on a straight line}\label{non-line}

We now consider the case when  $F_0(v)$ is not
concentrated on a straight line. In this case, without loss of generality,
we can assume that there exist three small balls denoted by
$A_i=B(\vb_i,\delta)$ with center at $v={\vb}_i$ and radius $\delta>0$ such that
$\int_{A_i} dF_0(v)=m_i>0$, for $i=1,2,3$. Up to a linear coordinate
transform, we can assume $\vb_1=\mathbf{0}$, $\vb_2$ and $\vb_3$ are linearly independent.
That is
$$
\eta_0=1-\big| \frac{\vb_2}{|\vb_2|}\cdot \frac{\vb_3}{|\vb_3|}\big|=1-|\cos\alpha|>0,
$$
where $\alpha$ is the angle
between $\vb_2$ and $\vb_3$.
Take two positive constants  $d_1 < d_2$ such that
\begin{align*}
0<  d_1 \min\{|\vb_2|, |\vb_3|\}<
d_2 \max\{ |\vb_2|, |\vb_3|\}\le \frac{\pi}{2}\,.
\end{align*}
Put $d = (d_1 +d_2)/2$. Firstly, we assume that $\xi^-$  varies on  the circle 
\begin{equation}\label{circle}
 \cC = \{ \xi \in \RR^3 \, ; \, |\xi| = d, \enskip \xi \, \bot\, (\vb_2 \times \vb_3)\,\}.
\end{equation}
In the following
discussion, we choose $\delta>0$ to be sufficiently small.
\begin{figure}[ht]\label{f-1}
\centering
\includegraphics[width=7cm]{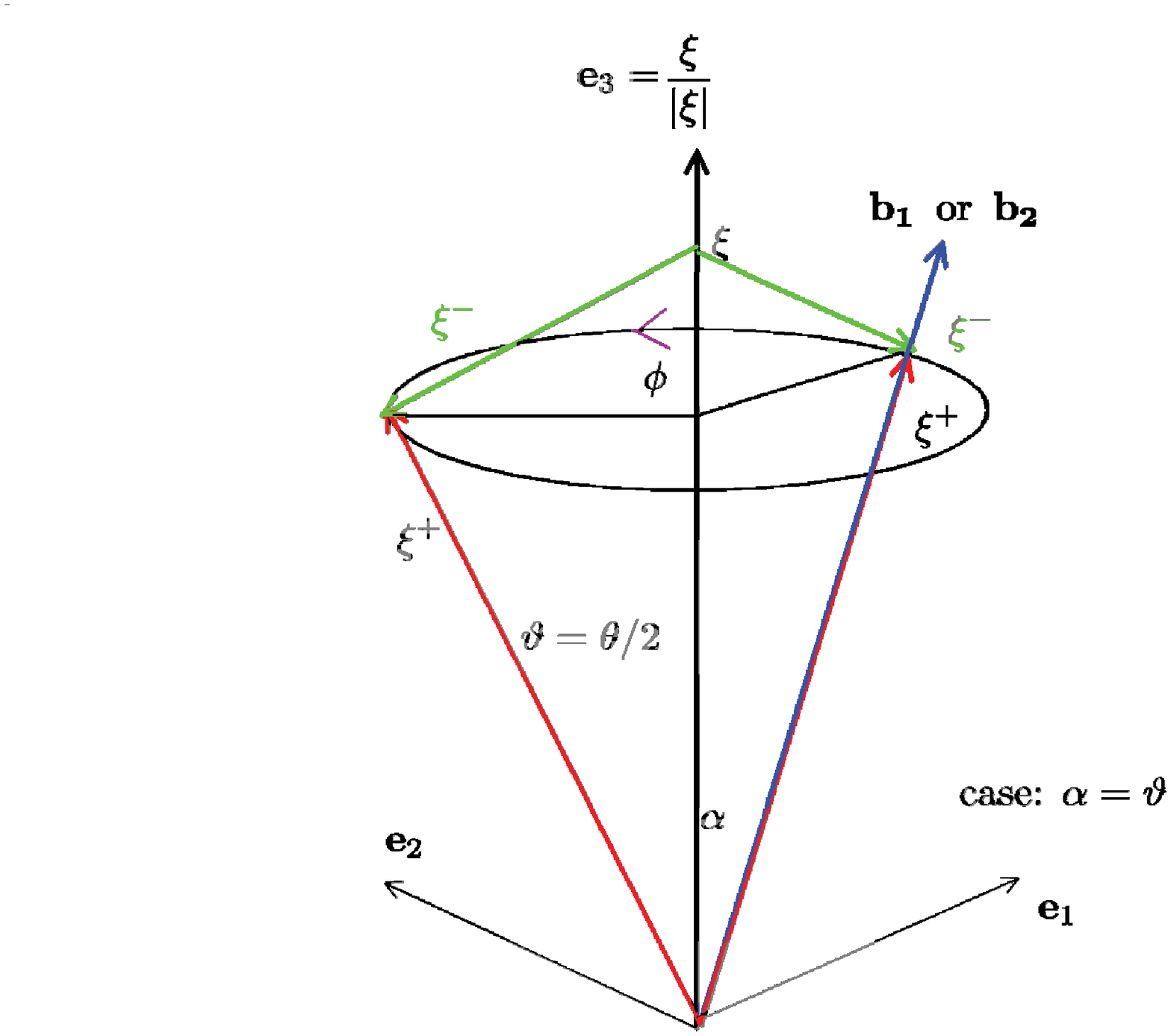}
\caption{$\xi^-$ and three vectors $\vb_1, \vb_2, \vb_3$ }
\end{figure}

Denote
$$
\int_{A_j} e^{-iv\cdot\xi^-} dF(v)=m_j(a_j+ i b_j), \qquad j=1,2,3.
$$
Note that $|a_j+i b_j|\le 1$. With the above notations, it is straightforward
to check that
\begin{eqnarray*}
&&(a_1,b_1)= (1,0)+ \ve_1, \quad \\
&&
(a_2,b_2)= (\cos(|\xi^-||\vb_2|\cos\gamma_1), 
\sin(|\xi^-||\vb_2|\cos\gamma_1)) + \ve_2,\\
&&(a_3,b_3)= (\cos(|\xi^-||\vb_3|\cos\gamma_2), 
\sin(|\xi^-||\vb_3|\cos\gamma_2)) + \ve_3,
\end{eqnarray*}
where $\gamma_1$ is the angle between the vectors $\xi^-$ and
$\vb_2$,
$\gamma_2$ is the angle between the vectors $\xi^-$ and $\vb_3$, 
$|\ve_i|=0(1)\delta$, $i=1,2,3$. 
Notice that
$\gamma_2=\gamma_1 \pm \alpha$.
With the above choice of parameters, we have when $\delta$ is sufficiently
small, 
\begin{eqnarray*}
&&2-\big|\frac{(a_1, b_1)}{|(a_1,b_1)|}\cdot \frac{(a_2, b_2)}{|(a_2,b_2)|}\big|
-\big|\frac{(a_1, b_1)}{|(a_1,b_1)|}\cdot \frac{(a_3, b_3)}{|(a_3,b_3)|}\big|\\
&&=2-\cos(|\xi^-||\vb_2|\cos\gamma_1 )-\cos(|\xi^-||\vb_3|\cos(\gamma_1 \pm\alpha)) 
+0(1)\delta\\
&&\ge c_0\eta_0,
\end{eqnarray*}
where $c_0>0$ is a constant independent of $\delta$.
Hence, if $\psi_0(\xi) = \int e^{-iv\cdot \xi} dF_0(v)$ and $\xi^-$ varies on 
$\cC$ defined by \eqref{circle}, then we have
\begin{align}\label{most}
&\psi_0(0)-|\psi_0(\xi^-)|=1-|\int_{A^c\cup_{j=1}^{3}A_j} e^{-i v\cdot \xi^-} d F_0(v)|\\
&\ge \sum_{j=1}^{3}\int_{A_j} dF_0(v) -|\sum_{j=1}^{3}\int_{A_j}  e^{-iv\cdot \xi^-} d F_0(v)|\notag \\
&=\sum_{j=1}^{3} m_j -|\sum_{j=1}^{3} m_j(a_j+i b_j)|\notag \\
&\ge \min\{m_1,m_2,m_3\}\Big( 3- |\sum_{j=1}^{3}(a_j +i b_j)|\Big) \notag \\
&\ge \frac{1}{3}\min\{m_1,m_2,m_3\}\Big\{2-\big|\frac{(a_1, b_1)}{|(a_1,b_1)|}\cdot \frac{(a_2, b_2)}{|(a_2,b_2)|}\big|
-\Big|\frac{(a_1, b_1)}{|(a_1,b_1)|}\cdot \frac{(a_3, b_3)}{|(a_3,b_3)|}\big|\Big\} \notag \\
&\ge   \frac{1}{3}\min\{m_1,m_2,m_3\}            c_0 \eta_0 := \kappa_0 , \notag
\end{align}
because $|a_j +ib_j| \le 1$ and 
\begin{align*}
&|\sum_{j=1}^{3}(a_j +i b_j)|^2 \le \Big(|a_1 +i b_1| + \sum_{j=2}^3 |a_j + i b_j| 
\Big|\frac{(a_1, b_1)}{|(a_1,b_1)|}\cdot \frac{(a_j, b_j)}{|(a_j,b_j)|}\Big|\Big)^2\\
&\qquad \qquad \qquad \qquad +\Big(\sum_{j=2}^3 |a_j + i b_j| 
\Big|\frac{(a_1, b_1)}{|(a_1,b_1)|}\times \frac{(a_j, b_j)}{|(a_j,b_j)|}\Big|\Big)^2\\
&\qquad \qquad \le \Big(1 + \sum_{j=2}^3 
\Big|\frac{(a_1, b_1)}{|(a_1,b_1)|}\cdot \frac{(a_j, b_j)}{|(a_j,b_j)|}\Big|\Big)^2
+\Big(\sum_{j=2}^3 
\Big|\frac{(a_1, b_1)}{|(a_1,b_1)}\times \frac{(a_j, b_j)}{|(a_j,b_j)|}\Big|\Big)^2\\
&\qquad \qquad \le \, 5 + 2 \sum_{j=2}^3 
\Big|\frac{(a_1, b_1)}{|(a_1,b_1)|}\cdot \frac{(a_j, b_j)}{|(a_j,b_j)|}\Big|\,\, .
\end{align*}

Since $\psi(t, \xi)$ is continuous (see Theorem \ref{remark-ck} in the Appendix) and $\psi(0,\xi) = \psi_0(\xi)$, by means of \eqref{most}, 
there exist $\mu >0$,  $\varepsilon >0$ 
and $T >0$ such that 
for any $\xi^-$ belonging to the set 
\begin{equation}\label{domain}
 \cC_{\mu, \varepsilon} = \{ \eta \in \RR^3 \, ; \, d - \mu \le |\eta| \le d + \mu, \enskip
\left| \frac{\eta}{|\eta|} \cdot \Big(\frac{\vb_2 \times \vb_3}{|\vb_2 \times \vb_3|}\Big)\right| \le \varepsilon \,\},
\end{equation}
we have 
\begin{align} \label{key-one}
1  - |\psi(t, \xi^-)| \ge \kappa_0/2 \enskip \enskip \mbox{for }\enskip 
t \in [0,T]\,.
\end{align}
Take a  $R >0$ such that $(d + \mu)/R = \varepsilon/10$. Let $|\xi| \ge R$, and for 
$\omega =\xi/|\xi| \in \SS^2$ take the coordinate
$\sigma = (\theta, \phi) \in [0,\pi/2]\times [0, 2\pi]$ with the pole $\omega$.
Write 
\[
\xi^- = \frac{\xi}{2} - \frac{|\xi|}{2} \sigma = \xi^-(\theta, \phi). \]
If $\theta$ satisfies 
\[
d-\mu \le |\xi^-(\theta, \phi)|= |\xi| \sin \frac{\theta}{2} \le d + \mu,
\]
then 
there exists an interval  $I_\omega \subset [0,2\pi]$ such that $\xi^-(\theta, \phi) 
\in \cC_{\mu, \varepsilon}$ for $\phi \in I_\omega$ because
$\theta/2 \le \sin^{-1} (d+\mu)/R < \varepsilon /5$ and the set
\[\{ \lambda \xi^{-}(\theta,\phi) \in \RR^3 \,;\, \phi \in [0,2\pi], 0\le \lambda \le 1 \}
\]
intersects the plane spanned by $\vb_2$ and $\vb_3$ when 
$|\omega \cdot (\vb_2 \times \vb_3)/|\vb_2 \times \vb_3|| < \cos \theta/2$
(see Figure 2). 

\begin{figure}[ht]\label{f-2}
\begin{center}
\includegraphics[width=10cm]{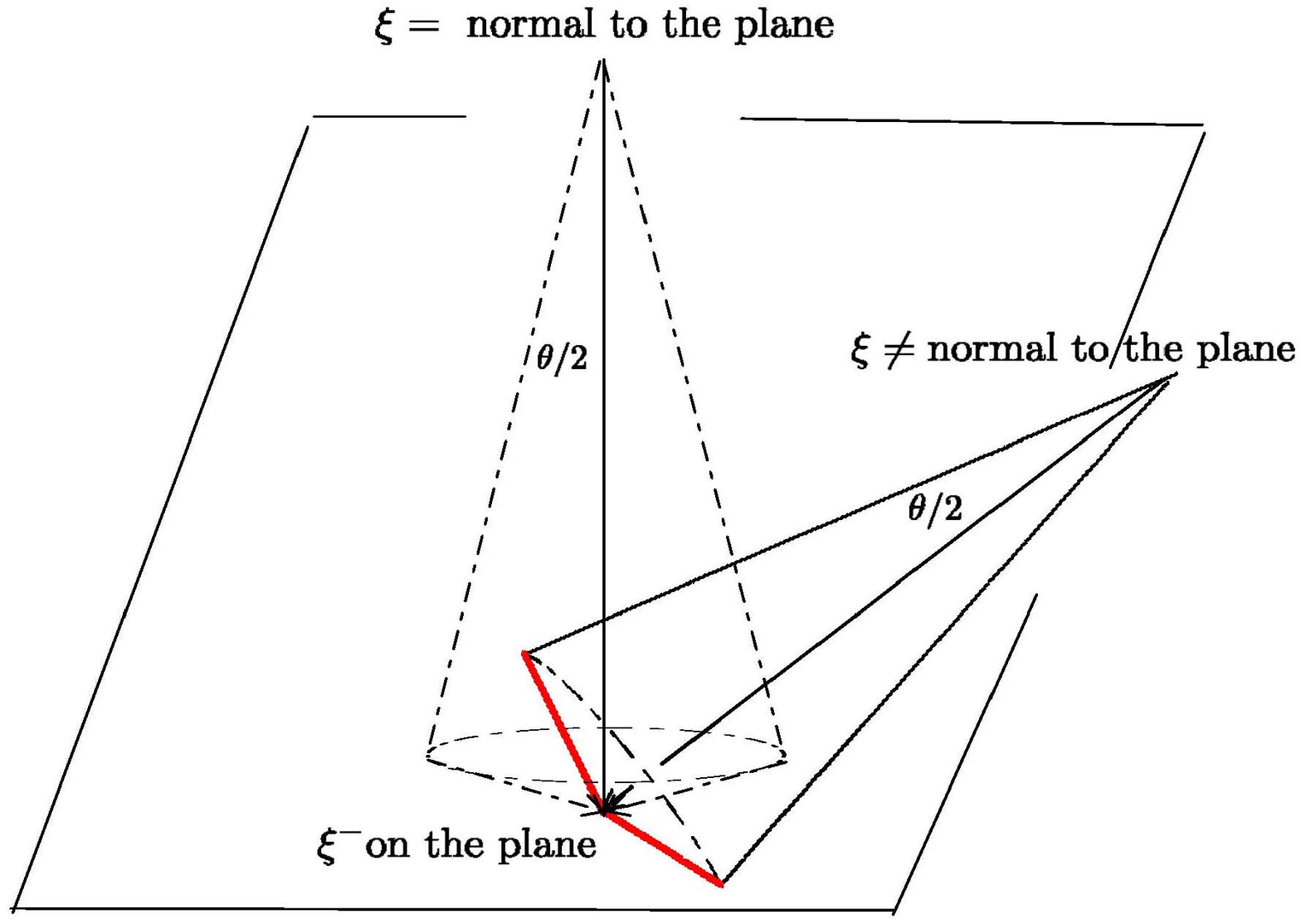}
\caption{Intersection between $\{\xi^-\}$ and the plane spanned by $\vb_2, \vb_3$}
\end{center}
\end{figure}

It is obvious that the interval $I_\omega$ plays the same role for $\tilde \omega \in \SS^2$ close to 
$\omega$. Therefore, 
for 
any $\xi$ belonging to a conic neighborhood of $\omega$
\[
\Gamma_\omega = \{ \xi \in \RR^3\,;\, 
\Big| \frac{\xi}{| \xi|} -\omega\Big|< \varepsilon_\omega ,\, |\xi| \ge R \}
\]
with a sufficiently small $\varepsilon_\omega >0$, we have

\begin{align*}
&\int_{\RR^3}\Big(\int_{\SS^2} b \Big(\frac{\xi}{|\xi|}\cdot \sigma\Big)
\Big(1 - |\psi(t, \xi^-|)\Big)d\sigma\Big)|h(\xi)|^2 d\xi\\
&\gtrsim \int_{\Gamma_\omega}\Big(\int_{I_\omega} d\phi \int_{2\sin^{-1}(d-\mu)/|\xi|}
^{2 \sin^{-1}(d +\mu)/|\xi|}\theta^{-1-2s} \frac{\kappa_0}{2} d \theta \Big )|h(\xi)|^2 d\xi \\
&\gtrsim  \int_{\Gamma_\omega} |\xi|^{2s}|h(\xi)|^2 d\xi,
\end{align*}
which together with the standard covering argument on $\SS^2$ yields
\begin{align*}
&\int_{\RR^3}\Big(\int_{\SS^2} b \Big(\frac{\xi}{|\xi|}\cdot \sigma\Big)
\Big(1 - |\psi(t, \xi^-|)\Big)d\sigma\Big)|h(\xi)|^2 d\xi\\
&\qquad  + \int_{\RR^3}|h(\xi)|^2 d\xi \gtrsim 
\int_{\RR^3} \langle \xi \rangle^{2s}|h(\xi)|^2 d\xi,
\end{align*}
if $t \in [0,T]$.


\subsection{Initial measure concentrated on a straight line}\label{line}
We now consider the case when  $F_0(v)$ is 
concentrated on a straight line and not equal to  a single Dirac measure.  By means of a suitable choice of the coordinate 
we may assume that $F_0(v) = \delta(v') F_{03}(v_3)$ and its Fourier transform $\psi_0(\xi) = \psi_{03}(\xi_3)$, where  $\psi_{03}$ is the Fourier transform of $F_{03}$. 
Since $F_{03}(v_3)$ is not a point Dirac measure in $\RR$, it follows from Corollary 3.5.11 in \cite{jacob} that
there exists a $\xi_{03} >0$ such that $|\psi_{03}(\pm \xi_{03})| <1$, in view of $\psi(-\xi) = \overline{\psi(\xi)}$. 
By means of the continuity of $\psi$, there exist $0< \kappa < 1$ and  $0 < a_1 < a_2 $  such that 
\begin{equation}\label{key}
|\psi_0(\xi',  \xi_3)| \le 1 - \kappa, \enskip  \enskip \mbox{for} \enskip \forall \xi' \in \RR^2, \enskip \forall \xi_3\in \RR \enskip \mbox{with } \enskip 
a_1 \le |\xi_3| \le a_2 \,.
\end{equation}

We now split the discussion into two cases.

\subsubsection{The case when $\xi^-$ is almost orthogonal to the third axis }

For the sake of simplicity, we denote $\xi^-$ by $\xi$ throughout this subsection except for the case when  confusion might occur. We also denote $\psi$ instead of $\psi_0$ for brevity.

Note that
\begin{align*}
(\pa_t |\psi|^2)(0,\xi) & = 2\,  \mbox{Re} \int_{\SS^2} b\Big(\frac{\xi}{|\xi|}\cdot \sigma\Big) 
\Big(\psi^+\psi^- \overline{\psi} - |\psi|^2 \Big) d\sigma\\
& = -  \int_{\SS^2} b\Big(\frac{\xi}{|\xi|}\cdot \sigma\Big)  \Big( |\psi^+|^2 +|\psi|^2 - 2\,  \mbox{Re} \{\psi^- \psi^+  \overline{\psi}\} \Big)d \sigma\\
&\qquad + \int_{\SS^2} b\Big(\frac{\xi}{|\xi|}\cdot \sigma\Big) \Big( |\psi^+|^2 - |\psi|^2\Big) d \sigma\\
& \le - \int_{\SS^2} b\Big(\frac{\xi}{|\xi|}\cdot \sigma\Big)  \Big(1- \psi^-\Big ) \Big( |\psi^+|^2 +|\psi|^2 \Big) d\sigma\\
&\qquad + \int_{\SS^2} b \Big(\frac{\xi}{|\xi|}\cdot \sigma\Big) \Big( |\psi^+|^2 - |\psi|^2\Big) d \sigma.
\end{align*}
If we put $\xi = \lambda \ve_2$ ($\lambda >0$) in the above estimate and take the 
polar coordinate $\sigma = (2\beta, \chi) \in [0,\pi/2]
\times [0,2\pi]$, (where $\chi$ starting from $\xi_3 =0$, see Figure 3) then
\begin{align*}
(\pa_t |\psi|^2)(0,\lambda \ve_2 ) \le - 2\int_0^{\pi/4} d \beta \int_0^{2\pi}d \chi
b(\cos 2 \beta) (\sin 2\beta)\Big (1- |\psi_3(\lambda \cos \beta \sin \chi)| \Big),
\end{align*}
because $\psi(\lambda \ve_2) =1$ and $|\psi|\le 1$. 
\begin{figure}[ht]\label{f-3}
\begin{center}
\includegraphics[width=8cm]{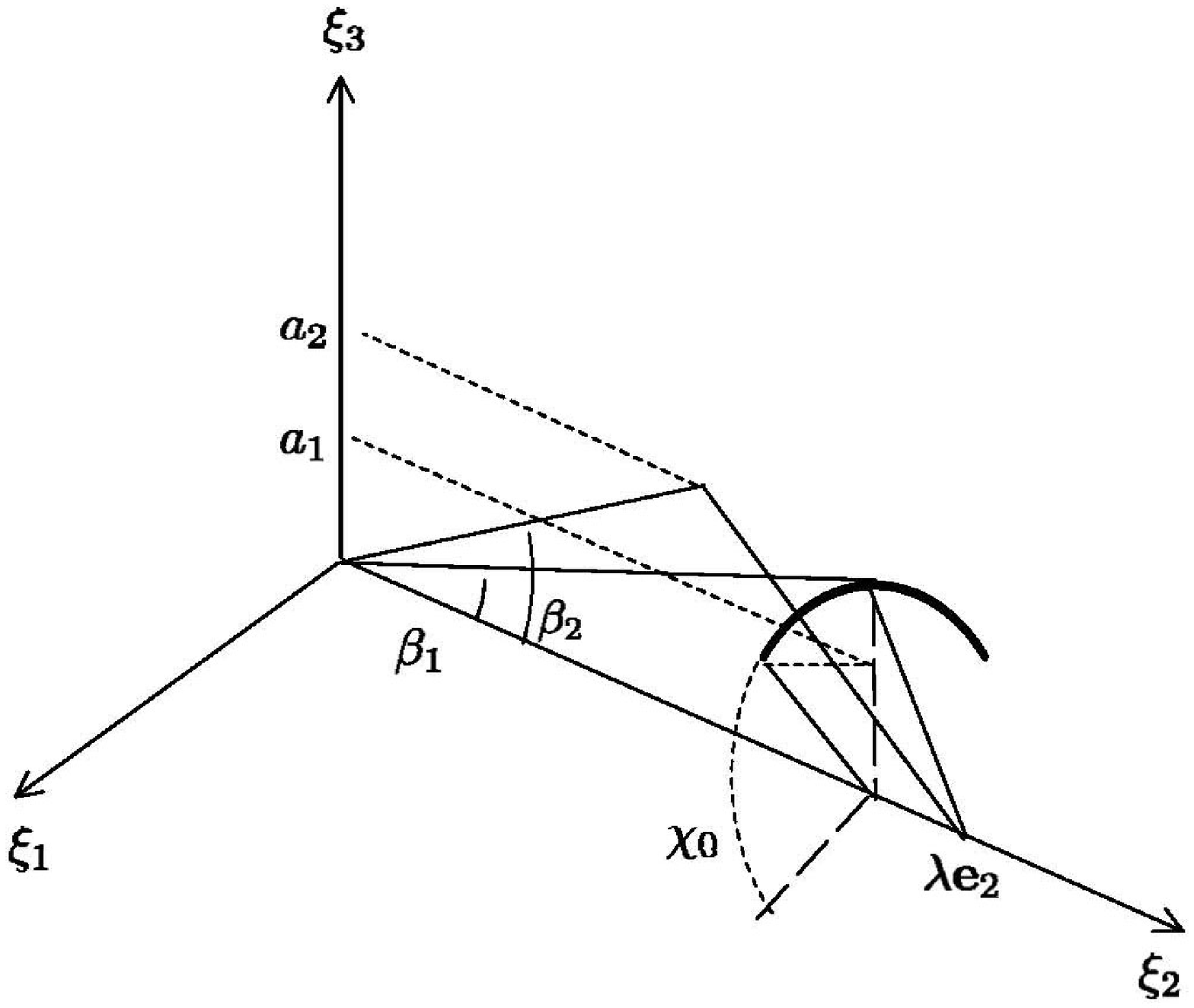}
\caption{$\xi^-$ with $\beta =\beta_1$ and $\chi \in [\chi_0, \pi - \chi_0]$}
\end{center}
\end{figure}
Choose $0< \beta_1 < \beta_2 < \pi/4$, $\chi_0 \in (0, \pi/2)$ and $\lambda >0$ such that
\begin{align*}
&\lambda \cos \beta_2 \sin \chi_0 =a_1 \,, \enskip \lambda \cos \beta_1 = a_2\,,\\
& 2\int_{\beta_1}^{\beta_2} b(\cos 2 \beta) (\sin 2\beta) d\beta = c_0 >0\,.
\end{align*}
Then it follows from \eqref{key} that
\begin{align*}
(\pa_t |\psi|^2)(0,\lambda \ve_2 ) \le - 2\int_{\beta_1}^{\beta_2} d \beta \int_{\chi_0}
^{\pi - \chi_0}d \chi
b(\cos 2 \beta) (\sin 2\beta)\kappa = - \kappa c_0 (\pi - 2 \chi_0)\,.
\end{align*}
Since $\psi$ is symmetric around $\xi_3$ axis, we have
\[
(\pa_t |\psi|^2)(0,\xi ) \le  - \kappa c_0 (\pi - 2 \chi_0), \enskip \mbox{if}\enskip
\xi \cdot \ve_3 =0 \enskip \mbox{and} \enskip |\xi| = \lambda.
\]
If we set $c_1 = \kappa c_0 (\pi - 2 \chi_0)$, then there exist $\varepsilon >0$, $T >0$ and $\delta >0$
such that
\begin{align*}
&(\pa_t |\psi|^2)(t,\xi) \le - c_1/2, \enskip\\
&\mbox{when $(t, \xi) \in  [0,T] \times \Big\{ \xi \in \RR^3\,;\, \big| |\xi|-\lambda\big| \le \delta, 
\left|\frac{\xi}{|\xi|} \cdot \ve_3\right| \le 2 \varepsilon \Big \}:=[0,T] \times \Gamma $,}
\end{align*}
because of the continuity of $\psi$ and $\pa_t \psi$ (see Theorem \ref{remark-ck} in the 
Appendix).

In what follows we use the notation $\xi^- = (\xi - |\xi| \sigma)/2$ to obtain the microlocal
time degenerate coercivity estimate.  If $(t, \xi^-)$ belongs to the  region 
$[0,T] \times \Gamma$, then it follows from the mean value theorem that there
exists
a $\rho \in(0,1)$ such that
\begin{align*}
1 - |\psi(t,\xi^-)| &\ge \frac{1- |\psi(t,\xi^-)|^2}{2}
= \frac{1}{2}
\Big(1 - |\psi(0,\xi^-)|^2 - (\pa_t |\psi|^2)(\rho t,\xi^- ) t \Big)\\
& \ge \frac{c_1}{4} t\,.
\end{align*}
Set $R_0 = (\lambda + \delta)/\varepsilon$ and 
\begin{equation}\label{paralel}
\Omega_0 = \{\xi \in \RR^3\,; \, |\xi| \ge R_0, \, \left| 1- \Big|\frac{\xi}{|\xi|}\cdot \ve_3 \Big| \right| 
\le \frac{2 \varepsilon^2}{\pi^2}\}\, \enskip (\mbox{see Figure 4}).
\end{equation}
If $\sigma = (\theta, \phi)$, we notice that $|\xi^-| = |\xi|\sin(\theta/2)$.
Moreover,
the fact that $\xi \in \Omega_0$ and $\sin \frac{\theta}{2} \le (\lambda +\delta)/|\xi|$
implies $\left|\frac{\xi^-}{|\xi^-|} \cdot \ve_3\right| \le 2 \varepsilon $.
\begin{figure}[ht]\label{f-4}
\begin{center}
\includegraphics[width=6cm]{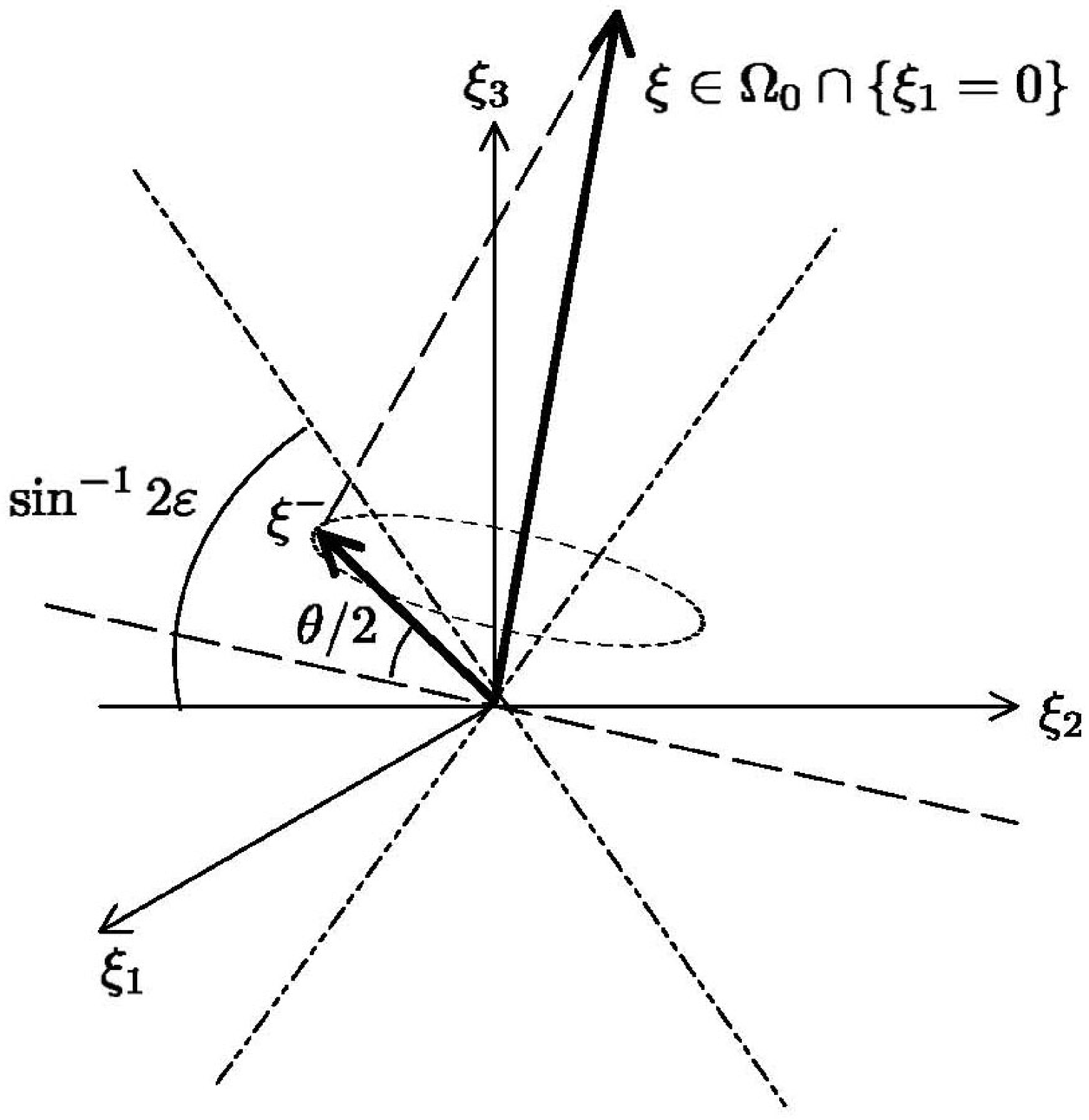}
\caption{$\xi \in \Omega_0$ and $\xi^-$ almost orthogonal to $\xi_3$}
\end{center}
\end{figure}
Therefore,
if $t \in [0,T]$ and if $h(\xi) \in L^2_s (\RR^3)$, then we have the micolocal
coercivity estimate in $\Omega_0$
\begin{align*}
&\int_{\RR^3} \Big(\int_{\SS^2}b\Big(\frac{\xi}{|\xi|}\cdot \sigma\Big) (1 - |\psi(t,\xi^-)| ) d\sigma\Big)
|h(\xi)|^2 d \xi\\
&\gtrsim \int_{\Omega_0} \Big(\int_{2 \sin^{-1}(\lambda-\delta)/|\xi|}
^{2 \sin^{-1}(\lambda+\delta)/|\xi|}\theta^{-1-2s} \frac{c_1t}{4}d\theta \Big)|h(\xi)|^2 d \xi\\
& \gtrsim t \int_{\Omega_0} |\xi|^{2s}|h(\xi)|^2 d \xi\,.
\end{align*}

\subsubsection{The microlocal coercivity estimate in $\Omega_0^c$}
In this subsection, we consider the case when $\xi$ belongs to
\[
\Omega_1= \Big\{\xi\in \RR^3\,; \, 
\left| 1- \Big|\frac{\xi}{|\xi|}\cdot \ve_3 \Big| \right| 
> \frac{2 \varepsilon^2}{\pi^2}
\Big\} \subset \Omega_0^c.
\]

\begin{figure}[ht]\label{f-5}
\begin{center}
\includegraphics[width=9cm]{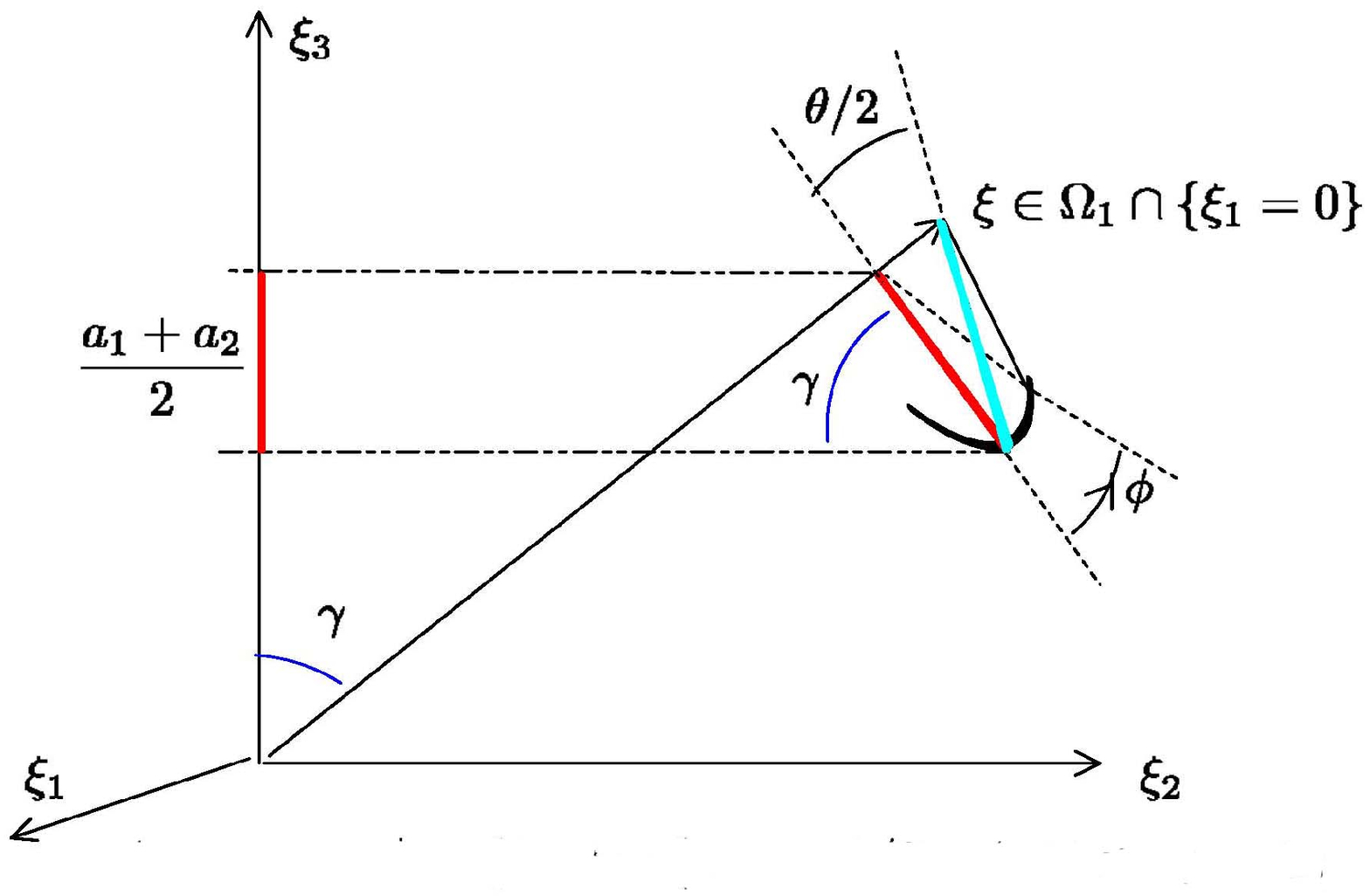}
\caption{$\mbox{Case:}\, |\xi^-\cdot 
{\bf e}_3|\cos( \theta/2 )=
(a_1+a_2)/2
\,\,\mbox{and} \,\, \xi^- \in \{\xi_1=0\}$}
\end{center}
\end{figure}

Fix an arbitrary $\omega \in \SS^2 \cap \Omega_1\cap \{ \omega\cdot \ve_3 \ge 0\}$.
Take a $\lambda >0$ such that $\lambda \sin \gamma = (a_1 + a_2)/2$, where $\gamma > 2 \varepsilon /\pi$
is the angle
between $\omega$ and $\ve_3$. 
If we take the polar coordinate $\sigma = (\theta, \phi) \in [0,\pi/2]\times [0,2\pi]$
with the pole $\omega = \xi/|\xi|$ and $\phi$ starting from the plane $\xi_1 =0$
(see Figure 5), then we have
\begin{equation}\label{est-deg}
\xi^- \cdot \ve_3 = |\xi^-|\big(\cos\frac{\theta}{2} \cos \phi \sin \gamma + \sin \frac{\theta}{2} 
\cos \gamma \Big),
\end{equation}
where $\xi^- = (\xi - |\xi|\sigma)/2$.
There exist
$\delta = \delta_\omega >0$, $\phi_\omega \in (0, \pi/4]$ and $\theta_\omega \in [0,\pi/4]$
such that
\begin{align}\label{projection}
&a_1 < (\lambda - \delta) \cos( \theta_\omega/2)\cos \phi_\omega \sin \gamma
 \\
& < (\lambda + \delta) \Big(\sin \gamma +  \tan \theta_\omega/2\Big) < a_2. \notag
\end{align}
Put $R_\omega \sin \theta_\omega/2 = \lambda + \delta_\omega$ and let  $\xi = |\xi|\omega$ with $|\xi| \ge R_\omega$. 
If $|\xi^-|= |\xi| \sin \theta/2 \in [\lambda -\delta, \lambda + \delta]$ and $|\xi| \ge R_\omega$, then
$\theta \le \theta_\omega$. Moreover, when $|\phi| \le \phi_\omega$ we have
\begin{align}\label{favorable}
\xi^- \cdot \ve_3 \in (a_1,a_2).
\end{align}
Since \eqref{projection} still holds  for other $\tilde \gamma$ close to $\gamma$, we have
\eqref{favorable} for 
any $\xi$ belonging to a conic neighborhood of $\omega$
\[
\Gamma_\omega = \{ \xi \in \RR^3\,;\, 
\Big| \frac{\xi}{| \xi|} -\omega\Big|< \varepsilon_\omega ,\, |\xi| \ge R_\omega \},
\]
with a sufficiently small $\varepsilon_\omega >0$, if $(\theta, \phi)$ varies in the same region
as above. Since $\psi(t,\xi)$ is continuous, it follows from \eqref{key}
that there exists a $T_\omega >0$ such that for any $ t \in [0,T_\omega]$ we have 
\[ |\psi(t,\xi^-)| \le 1-\frac{\kappa}{2}
\enskip \mbox{if} \enskip |\xi^-| \in [\lambda-\delta, \lambda+\delta]
\enskip \mbox{and} \enskip \xi^- \cdot \ve_3 \in [a_1,a_2].
\]
Therefore
\begin{align*}
&\int_{\RR^3}\Big(\int_{\SS^2} b \Big(\frac{\xi}{|\xi|}\cdot \sigma\Big)
\Big(1 - |\psi(t, \xi^-|)\Big)d\sigma\Big)|h(\xi)|^2 d\xi\\
&\gtrsim \int_{\Gamma_\omega}\Big(\int_{-\phi_{\omega}}^{\phi_\omega} d\phi \int_{2\sin^{-1}(\lambda-\delta)/|\xi|}
^{2 \sin^{-1}(\lambda +\delta)/|\xi|}\theta^{-1-2s} \frac{\kappa}{2} d \theta \Big )|h(\xi)|^2 d\xi \\
&\gtrsim  \int_{\Gamma_\omega} |\xi|^{2s}|h(\xi)|^2 d\xi.
\end{align*}
The estimation for $\omega \in \SS^2 \cap \Omega_1 \cap \{\omega \cdot \ve_3 \le 0\}$ is similar, so that we omit it for brevity.

\subsubsection{The conclusion}
By means of the covering argument, we have for a sufficiently large $R >0$ and a sufficiently small $T>0$,
\begin{align*}
&\int_{\RR^3} \Big(\int_{\SS^2}b\Big(\frac{\xi}{|\xi|}\cdot \sigma\Big) (1 - |\psi(t,\xi^-)| ) d\sigma\Big)
|h(\xi)|^2 d \xi\\
&\qquad \gtrsim t \int_{\{|\xi| \ge R\}} |\xi|^{2s}|h(\xi)|^2 d \xi\, ,\enskip t \in [0, T]\,.
\end{align*}
This together with the coercivity estimate obtained in the first
subsection concludes the proof of Lemma \ref{time-deg-thm}.

Before ending this subsection, we remark that if $\psi_0(\xi) =
\int e^{-iv\cdot \xi} dF_0(v)$,  then for a large 
$R>0$ we have the following degenerate coercivity estimate
\begin{align}\label{usual-deg}
&\int_{\RR^3} \Big(\int_{\SS^2}b\Big(\frac{\xi}{|\xi|}\cdot \sigma\Big) (1 - |\psi_0(\xi^-)| ) d\sigma\Big)
|h(\xi)|^2 d \xi\\
&\qquad \gtrsim  \int_{\{|\xi| \ge R\}} \big(|\xi_1|^2 +|\xi_2|^2 +| \xi_3|\big)^{s}|h(\xi)|^2 d \xi\, .\notag
\end{align}
Indeed, it follows from \eqref{est-deg} that
\[
\xi^- \cdot \ve_3 \sim |\xi| (\theta \gamma \cos \phi  + \theta^2)
\sim |\xi| \Big\{\theta \Big(\frac{|\xi_1|^2 +|\xi_2|^2}{|\xi|^2}\Big)^{1/2} \cos \phi  + \theta^2\Big\}
\] for
sufficiently small $\gamma$ and $\theta$, and we have
\begin{align*}
\int_{\SS^2}b\Big(\frac{\xi}{|\xi|}\cdot \sigma\Big) (1 - |\psi_0(\xi^-)| ) d\sigma
\gtrsim \kappa \int_{\cA} \theta^{-1-2s} d\theta d \phi\,,
\end{align*}
where $\cA = \{(\theta, \phi) \,;\, \xi^- \cdot \ve_3 \in [a_1, a_2]\}$.
However, this degenerate coercivity estimate is not sufficient to
show the smoothing effect because the continuity in $\psi(t,\xi)$
does not imply \eqref{usual-deg} with $\psi_0(\xi^-)$ replaced by $\psi(t,\xi^-)$.

\section{Appendix}\label{ap}
\setcounter{equation}{0}
In this appendix we {first} recall the result given in \cite{Cannone-Karch,morimoto-12}, 
and prove the continuity of $\pa_t \psi(t,\xi)$.
For this, assume 
\begin{align}\label{index-sing}
\exists \alpha_0 \in (0, 2]\enskip\mbox{ such that} \enskip  (\sin \theta/2)^{\alpha_0} b(\cos \theta) 
\sin \theta  \in L^1((0, \pi/2]),
\end{align}
which is fulfilled for $b(\cos \theta)$ with \eqref{1.2} if $2s < \alpha_0$.
As stated in the proof of Theorem \ref{main-thm} in the introduction,  
it follows from the Bobylev formula that the Cauchy problem \eqref{bol}-\eqref{initial}
is reduced to \eqref{c-p-fourier}, if 
$\psi_0(\xi) = \int_{\RR^3}e^{-iv\cdot \xi} dF_0(v)$ and $\psi(t,v)$ denotes
the Fourier transform of the probability measure solution. 


\begin{theo}\label{remark-ck}
Assume that $b(\cos \theta)$ 
satisfies \eqref{index-sing} for some $\alpha_0 \in (0,2]$.
Then for each $\alpha \in [\alpha_0, 2]$ and every $\psi_0 \in \cK^\alpha$ there
exists a classical solution $\psi \in 
C([0,\infty), \cK^\alpha)$ of the Cauchy problem \eqref{c-p-fourier}. The solution
is unique in the space $C([0,\infty), \cK^{\alpha_0})$. Furthermore, if $\alpha \in [\alpha_0, 2]$ and if
$\psi(t,\xi)$, $\varphi(t,\xi)$ $\in C([0,\infty), \cK^{\alpha})$ are two solutions to the Cauchy problem
\eqref{c-p-fourier} with initial data $\psi_0, \varphi_0 \in \cK^{\alpha}$, respectively, then 
for any $t >0$ we have
\begin{equation}\label{stability-0}
\|\psi(t) -\varphi(t)\|_\alpha \le e^{\lambda_\alpha t} \|\psi_0 -\varphi_0\|_\alpha ,
\end{equation}
where
\begin{equation}\label{constant-imp}
 \lambda_\alpha 
=2\pi \int_0^{\pi/2} b(\cos \theta)
\{\cos^\alpha \frac{\theta}{2 }+ \sin^\alpha \frac{\theta}{2}-1 \}
\sin \theta d \theta\,.
\end{equation}
Furthermore, $\partial_t \psi(t,\xi)$ is continuous in $[0,\infty)\times \RR^3$.
\end{theo}
The assumption \eqref{index-sing} with $\alpha = \alpha_0$ can be written as
\begin{equation}\label{integrability-1}
(1-\tau)^{\alpha_0/2}b(\tau) \in L^1([0,1))\,,
\end{equation}
by the change of variable $\tau = \cos \theta$. 
Theorem \ref{remark-ck} ameliorates Theorem 2.2 of \cite{Cannone-Karch}, where \eqref{integrability-1}
is  assumed with $\alpha_0/2$ replaced by $\alpha_0/4$, see (2.6) of \cite{Cannone-Karch}.
In what follows, we only prove the last statement of Theorem \ref{remark-ck} because other parts
are  already given in \cite{morimoto-12}.

\begin{proof}[Proof of the continuity of $\pa_t(t,\xi)$]
If we put  $\zeta = \left(\xi^+ \cdot \frac{\xi}{|\xi|}\right) \frac{\xi}{|\xi|}$ and 
consider $\tilde \xi^+ = \zeta -(\xi^+-\zeta)$, 
(which is
symmetric to $\xi^+$ on $\SS^2$, see Figure 6) as in \cite{morimoto-12}, then 
the first equation of \eqref{c-p-fourier} can be written as
\begin{align}\label{decompo}
\pa_t\psi(t,\xi)
&= \frac{1}{2}\int_{\SS^2}b\left(\frac{\xi \cdot \sigma}{|\xi|}\right) 
\Big( \psi(t,\xi^+) + \psi(t, \tilde \xi^+) - 2 \psi(t,\zeta) \Big)d\sigma  \notag \\
&\quad +\int_{\SS^2}b\left(\frac{\xi \cdot \sigma}{|\xi|}\right)
\Big(\psi(t, \zeta) - \psi(t,\xi) \Big) d\sigma  \\
&\quad + \int_{\SS^2}b\left(\frac{\xi \cdot \sigma}{|\xi|}\right) \psi(t,\xi^+)
\Big(\psi(t,\xi^-) -1\Big) d\sigma \notag \\
& = I_1(t,\xi) + I_2(t,\xi) + I_3(t,\xi)\,. \notag
\end{align}
Putting $\eta^+ = \xi^+ - \zeta$, we have, under the notation
$d F_t(v) = f(t,v)dv$,
\begin{align*}
&|\psi(t, \xi^+) + \psi(t,\tilde \xi^+) - 2 \psi(t,\zeta) |=\left
|\int_{\RR^3} e^{-i\zeta \cdot v}\Big( e^{-i \eta^+ \cdot v} + e^{ i \eta^+\cdot v}  - 2\Big)d F_t(v) \right|\\
&\qquad \le \int_{\RR^3}
|e^{-i \zeta \cdot v}| \Big( 2 -e^{-i\eta^+ \cdot v} - e^{i\eta^+ \cdot v}\Big) dF_t(v)\\
&\qquad = 2 - \psi(t, \eta^+) - \psi(t, -\eta^+)\\
&\qquad \le 2\|1-\psi(t)\|_\alpha |\eta^+|^\alpha \le  
2 e^{\lambda_\alpha t} \|1-\psi_0\|_\alpha \Big( |\xi| \sin(\theta/2)\Big)^\alpha\,,
\end{align*}
because $|\eta^+| = |\xi^+| \sin(\theta/2)$ and \eqref{stability-0} with
$\varphi_0 =\varphi(t) =1$.
\begin{figure}\label{f-6}
\begin{center}
\includegraphics[width=2.5in]{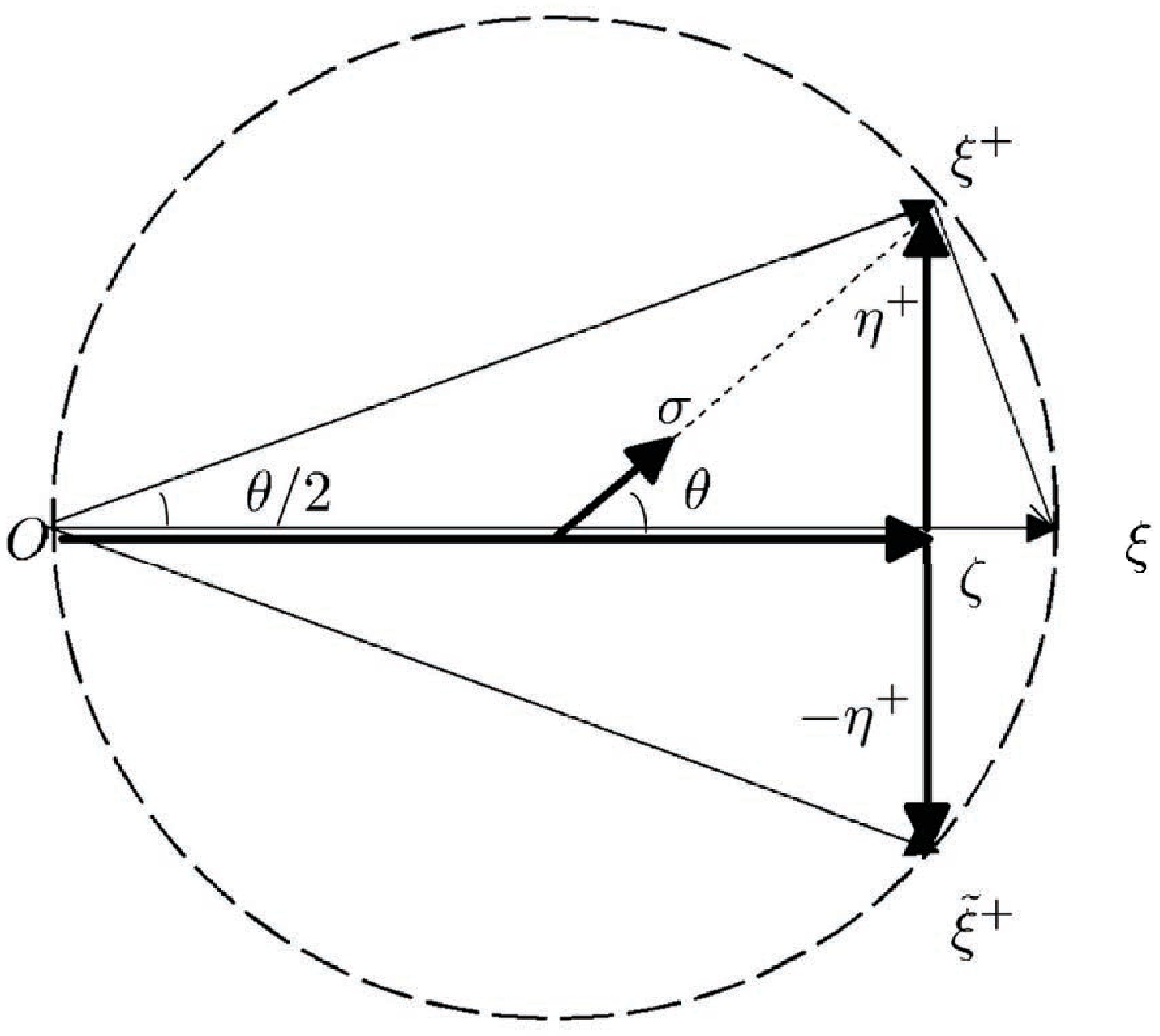}
\caption{$\cos \theta = \frac{\xi}{|\xi|} \cdot \sigma$, \enskip $\eta^+ =\xi^+ -\zeta$}
\end{center}
\end{figure}
Hence 
\[
|I_1(t,\xi)| \le 4 \pi e^{\lambda_\alpha t}
 \|1-\psi_0\|_\alpha |\xi|^\alpha \int_0^{\pi/2}\sin^\alpha (\theta/2) b(\cos \theta)\sin \theta 
d \theta\,,
\]
which together with the Lebesgue convergence theorem shows
\[
\lim_{(t,\xi) \rightarrow (t_0,\xi_0)}
I_1(t,\xi) = I_1(t_0,\xi_0)\,.
\]
In order to show similar estimates hold 
for $I_2, I_3$, we recall (19) of Lemma 2.1 in 
\cite{morimoto-12}, that is, the fact that if $\varphi \in \cK^\alpha$ then we have
\begin{align}\label{hoelder}
|\varphi(\xi) - \varphi(\xi+\eta) |\le \|\varphi -1\|_\alpha\Big(4|\xi|^{\alpha/2}|\eta|^{\alpha/2} +|\eta|^\alpha\Big)\enskip \mbox{for all $\xi, \eta \in \RR^3$}.
\end{align}
Thanks to this with $\eta = \zeta -\xi$,  
\[
|I_2(t,\xi)| \le 10 \pi  e^{\lambda_\alpha t} \|1-\psi_0\|_\alpha |\xi|^\alpha \int_0^{\pi/2}\sin^\alpha (\theta/2) b(\cos \theta)\sin \theta 
d \theta\,,
\]
because $|\zeta -\xi| = |\xi|\sin^2(\theta/2)$. 
Note that similar estimate holds for $I_3$.  Hence, we obtain 
the continuity of $\pa_t \psi(t,\xi)$. 
\end{proof}

{
\begin{prop}\label{energy-con}
Assume that $b(\cos \theta)$ 
satisfies \eqref{index-sing} for some $\alpha_0 \in (0,2]$.
If $F_0 \in P_2(\RR^3)$ then the unique measure solution $F_t(v)
\in C([0,\infty), \tilde P_2)$ belongs to $P_2(\RR^3)$ for each $t>0$, more precisely,
\begin{equation}\label{energy}
\int |v|^2 dF_t(v) \le \int |v|^2 dF_0(v). 
\end{equation}
Furthermore, if $\alpha_0 \le 1$ then the equality holds,
that is, the energy is conserved. 
\end{prop}

\begin{proof}
As a standard practice, we consider the increasing sequence of bounded collision kernels
\begin{equation}\label{cutoff}
b_n(\cos \theta) = \min \{ b(\cos \theta), n\}
\end{equation}
and denote by $\psi_n(t,\xi)$ the solution in 
$C([0,\infty); \cK_2)$ to the Cauchy problem 
\eqref{c-p-fourier} with $b$ replaced by the cutoff $b_n$, for the same initial datum
$\psi_0(\xi) = \int e^{-iv\cdot \xi} dF_0(v)$.  
It follows from Lemma 2.2 of \cite{PT} that
\begin{align}\label{cut-off-ene}
\int |v|^2 dF_t^{(n)}(v) = \int |v|^2 dF_0(v)\,,
\end{align}
where $F_t^{(n)} = \cF^{-1} \psi_n(t,\cdot)$. As proven in
\cite{PT, Cannone-Karch, morimoto-12}, we have 
the equi-continuity of $\{\psi_n(t,\xi)\}$ on $[0,\infty) \times \{|\xi| \le R\}$ 
for any fixed $R>0$. Since $|\psi_n| \le 1$, the Ascoli-Arzel\'a theorem gives a 
convergent subsequence $\{\psi_{n_k}\}_{k=1}^\infty$ and the solution 
$\psi = \lim_{k \rightarrow \infty} \psi_{n_k}$. Take a $\chi(v)$ in $C_0^\infty (\RR^3)$ satisfying 
$0\le \chi \le 1$ and 
$\chi = 1$ on $\{|v| \le 1\}$. Since $\psi_{n_k}(t) \rightarrow \psi(t)$ in $\cS'(\RR^3)$ for
each $t >0$, it follows from \eqref{cut-off-ene} that for any $m \in \NN$
\[
\int  |v|^2 \chi\Big(\frac{v}{m}\Big) dF_t (v) = \lim_{k \rightarrow \infty}
\int  |v|^2 \chi\Big(\frac{v}{m}\Big) dF^{(n_k)}_t (v) \le  \int |v|^2 dF_0(v).
\]
Letting $m \rightarrow \infty$ we obtain \eqref{energy}.
In the mild singularity case, $\alpha_0 \le 1$, we can use Theorem 2 of 
\cite{lu-wennberg} and its proof to show the reverse inequality of \eqref{energy}.
\end{proof}

In the rest of this appendix, we consider the case when $b(\cos \theta)$ satisfies \eqref{1.2} with $0<s <1/2$. 
We first prove the propagation of the moment as follows.

\begin{prop}\label{propagation-moments}
Let $\alpha \ge 1$ and $b(\cos \theta)$ satisfy \eqref{1.2} with $0<s<1/2$. If the initial data $F_0 \in P_\alpha$, then 
the unique solution $f(t,v)$ belongs to $ P_\alpha$ for any $t>0$. More 
precisely, there exists a constant $C >0$ independent of $t$ such that
\begin{align}\label{prop-moment}
\int \la v \ra^\alpha f(t,v) dv \le C e^{Ct} \int\la v \ra^\alpha  dF_0(v)\,.
\end{align}
\end{prop}
\begin{proof}
For the simplicity of the notations, we consider the case where $F_0$ has a density function $f_0(v)$. 
For $\delta >0$, put 
\[W_\delta(v) = \frac{\la v \ra^\alpha}{1+\delta \la v\ra^\alpha}\,.
\]
Since $x/(1+\delta x)$ is increasing in $[1,\infty]$ and $|v'| \le |v| +|v_*|$,  we have
\begin{align*}
W_{\delta}(v') &\lesssim \frac{\la v\ra^\alpha + \la v_*\ra^\alpha}{1+\delta (\la v\ra^\alpha + \la v_*\ra^\alpha)}
= \frac{\la v\ra^\alpha }{1+\delta (\la v\ra^\alpha + \la v_*\ra^\alpha)}
+\frac{ \la v_*\ra^\alpha}{1+\delta (\la v\ra^\alpha + \la v_*\ra^\alpha)}\\
&\le \frac{\la v\ra^\alpha }{1+\delta \la v\ra^\alpha }
+\frac{ \la v_*\ra^\alpha}{1+\delta \la v_*\ra^\alpha} = W_{\delta}(v)+W_{\delta}(v_*)\,.
\end{align*}
Therefore,
$|W_{\delta}(v') - W_{\delta}(v)| \lesssim W_{\delta}(v)+W_{\delta}(v_*) $. If $b_n$ is the same as in  \eqref{cutoff} and $f^n(t,v)$ is a
unique solution of the corresponding Cauchy problem, then 
we have
\begin{align*}
\frac{d }{dt} \int f^n(t,v)W_{\delta}(v) dv \lesssim 
\Big(\int b_n d\sigma \Big) \Big(\int f^n(t,v)W_{\delta}(v) dv \Big) \Big(\int f^n(t,v_*)dv_* \Big)\,,
\end{align*}
which shows 
\[
\int f^n(t,v)W_{\delta}(v) dv \le C_n e^{C_n t} \int f_0(v)W_{\delta}(v) dv.
\]
Taking the limit $\delta \rightarrow +0$, we have $f^n(t,v) \in L^1_\alpha$. 
Note that
\[
\la v' \ra^\alpha - \la v \ra^\alpha \lesssim \big(\la v  \ra^{\alpha-1} +  \la v_*  \ra^{\alpha-1}\big)
|v-v_*|\sin \frac{\theta}{2}.\]
There exists a $C >0$ independent of $n$ such that 
\begin{align*}
\frac{d }{dt} \int f^n(t,v)\la v \ra^\alpha  dv &\le 
\Big(\int b_n \theta d\sigma \Big) \Big(\int f^n(t,v) \la v \ra^\alpha dv \Big) \Big(\int f^n(t,v_*)dv_* \Big)\\
&\le C \int f^n(t,v) \la v \ra^\alpha dv\,.
\end{align*}
We have 
$\int f^n(t,v)\la v \ra^\alpha  dv \le C e^{Ct} \int f_0(v)\la v \ra^\alpha  dv$.
For the same  $\chi(v)$ defined in the proof of Proposition \ref{energy-con},  we have 
$$\int f^n(t,v)\la v \ra^\alpha\chi\left(\frac{v}{m}\right)   dv \le C e^{Ct} \int f_0(v)\la v \ra^\alpha  dv. $$
Since $f^n(t,v)\rightarrow f(t,v)$ in $\cS'(\RR^3_v)$ and $\la v \ra^\alpha \chi\left(\frac{v}{m}\right)  \in \cS$, we get
 $$\int f(t,v)\la v \ra^\alpha \chi\left(\frac{v}{m}\right)     dv \le C e^{Ct} \int f_0(v)\la v \ra^\alpha  dv. $$
 Letting $m \rightarrow  \infty$ gives the desired estimate \eqref{prop-moment}. 
 \end{proof}

We now turn to  show the  boundedness of the entropy of the solution for the initial data $ f_0 \in L^1_{1}(\RR^3) \cap L \log L(\RR^3)$ in the mild singularity case
when $0<s<\frac 12$, by using the equivalence between
the metric $\|\cdot \|_\alpha$ and the weak $*$  topology,  
inspired by the proof of Theorem 1 in \cite{toscani-villani}. 
\begin{prop}\label{entropy-finite}
Let $b(\cos \theta)$ satisfy \eqref{1.2} with $0<s<1/2$. Assume that $0 \le f_0 \in L^1_1(\RR^3)  \cap  L \log L(\RR^3)$ and $\int f_0(v) dv =1$. 
If $f(t,v)$ is a unique solution in $C([0,\infty), \tilde P_1)$   of the Cauchy problem  \eqref{bol}-\eqref{initial}, then $f(t,v) \in L^1_1(\RR^3)  \cap  L \log L(\RR^3)$ for any $t >0$. 
Moreover, for any $T_1 >0$ there exists a $C_{T_1} >0$ such that 
\begin{align}\label{boundedness-234}
\sup_{0< t \le T_1} \int \la v \ra f(t,v)   dv + \int f(t,v) \log \big( 1 +  f(t,v) \big) dv \le C_{T_1}\,.
\end{align}
\end{prop}
\begin{proof}
By the previous proposition, there exists a $C>0$ such that 
\begin{align}
\int \la v \ra f(t,v) dv \le C e^{Ct} \int\la v \ra f_0(v) dv. 
\end{align}
For $m \in \NN$ and $\chi \in C_0^\infty$ in the proof of Proposition \ref{energy-con}, put $f_{0,m}(v) = c_m \chi\left(\frac{v}{m}\right) f_0(v)$ with  $c_m = 1/ \int \chi\left(\frac{v}{m}\right)f_0(v)dv$.
Since 
$f_0(v) \in L^1_1(\RR^3)$, there exists $M \in \NN$
such that  $1 \le  c_m \le 2$ for all $ m \ge M$.  Consider the  solutions $f_m(t,v)$ for the initial data $f_{0,m}(v)$ for $m \ge M$. Since $f_{0,m} \in L^1_2(\RR^3) \cap L \log L(\RR^3)$, we have 
by Theorem 1 of \cite{villani} that 
\begin{align}\label{H-theorem}
\int f_m(t,v) \log f_m(t,v) dv  + \int_0^t D(f_m)(s) ds =  \int f_{0,m}(v) \log f_{0,m}(v) dv\,, 
\end{align}
where $D(f)(t)$ is defined by 
\[
D(f) (t) = \frac{1}{4}\iiint_{\RR^3 \times \RR^3 \times \SS^2} b(f'_*f' -f_* f) \log \frac{f'_*f'}{ f_* f}dvdv_*d\sigma \ge 0
\]
with $f = f(t,\cdot)$. Therefore, writing $\chi_m = \chi(v/m)$, we have 
\begin{align*}
&\int f_m(t,v) \log f_m(t,v) dv  \leq \int f_{0,m}(v) \log f_{0,m}(v) dv\\
&\le 2\log 2 \int f_0 dv  + 2  \int \chi_m f_0 \log^+  \chi_m f_0  dv.
\end{align*}
Note that $x \log x \ge -y + x \log y$ $(x \ge 0, y >0)$. Hence, 
\begin{align*}
0 \ge f_m(t,v) \log^-f_{m}(t,v) \ge - e^{-\la v \ra} - \la v \ra   f_{m}(t,v).
\end{align*}
Thanks to Proposition \ref{propagation-moments} again, 
\[
\int f_m(t,v)\la v \ra  dv \le C e^{Ct} \int f_{0,m}(v)\la v \ra  dv \le  2 C e^{Ct} \int \la v \ra f_{0}(v) dv.
\]
Therefore 
\begin{align*}
&\int f_m(t,v) \log^+ f_m(t,v) dv  
\le 2 \log 2 \int f_0 dv \\
& \quad + 2 \int f_{0}(v) \log^+  f_{0}(v) dv  +\int e^{-\la v \ra} dv  +  C e^{Ct} \int  \la v \ra f_{0}(v) dv. 
\end{align*}
Thus, for any $ T>0$ there exists $C_T >0$ such that
\begin{align*}
\sup_{0 < t \le T}\Big( \int f_m(t,v)\la v \ra  dv+ \int f_m(t,v) \log (1+ f_m(t,v) )dv \Big)  \le C_T\,,
\end{align*}
which concludes the weak compactness of $\{ f_m \}$ in $L^1(\RR^3)$, by means of Dunford-Pettis criterion.

Note that $\hat f_{0,m}, \hat f_0 \in \cK^1$ uniformly with respect to $m$  where the condition $\int v_j f_{0,m}(v) dv =0$ is not required. 
In fact, 
\begin{align}\label{uni-norm-1}
\|\hat f_{0,m} -1\|_1 &\le  \int  \sup_{|\xi| \ne 0} \frac{|e^{-i \frac{v}{|v|} \cdot (|v|\xi)}- 1|}{\left |(|v|\xi )\right|} |v| f_{0,m}(v)dv \notag \\
&\le 2 \int |v| f_{0,m}(v)dv \le 4 \int |v|f_0(v)dv.
\end{align}
Take a constant $\beta > 0$ satisfying $2s < \beta <1$. Then,  for any $\delta >0$ we have 
\begin{align*}
\sup_{0<|\xi| < \delta} \frac{|\hat f_{0,m} (\xi) - \hat f_{0}(\xi)|}{|\xi|^\beta} &\leq  \delta^{1-\beta} \Big( \sup_{0<|\xi| < \delta} \frac{|\hat f_{0,m} (\xi) - 1|}{|\xi|}
+ \sup_{0< |\xi| < \delta} \frac{|1 - \hat f_{0}(\xi)|}{|\xi|}\Big)\\
&\le 6 \delta^{1-\beta} \int |v|f_0(v)dv\,.
\end{align*}
Note that for any  $R >0$ large enough, we have
\begin{align*}
\sup_{|\xi| >  R} \frac{|\hat f_{0,m} (\xi) - \hat f_{0}(\xi)|}{|\xi|^\beta} \le2 R^{-\beta} \,.
\end{align*}
In view of \eqref{uni-norm-1},  it follows from Lemma 2.1 of \cite{morimoto-12}  that $\{\hat f_{0,m}\}$ is uniformly equi-continuous on the compact set $\{\delta \le |\xi| \le R\}$.
By the Ascoli-Arzel\'a theorem, $\{\hat f_{0,m}\}$ is a Cauchy sequence in $\|\cdot \|_\beta$, taking a subsequence if necessary.  Since $f_{0,m}(v) \rightarrow f_0(v)$ in $\cS'(\RR^3)$ as $m$ tends to  $\infty$,
we concludes $\|\hat f_{0,m}-  \hat f_{0}\|_\beta \rightarrow 0$.

It follows from \eqref{stability-0} that
\[
\|\hat f_m(t,\cdot) - \hat f(t,\cdot)\|_\beta \le e^{\lambda_\beta t}\|\hat f_{0,m} - \hat f_0\|_\beta,
\]
which shows $\hat f_m(t,\xi) \rightarrow \hat f(t,\xi)$ everywhere in $\xi$ for any fixed $t>0$.
Since $|\hat f_m(t,\xi)|, |\hat f(t,\xi)| \le 1$, we know 
$f_m(t,v) \rightarrow f(t, v)$ in $\cS'(\RR^3)$.

If we recall the weak compactness of $\{f_m(t,\cdot)\}$ in $L^1(\RR^3)$, then we obtain 
\[
f_m(t,\cdot)  \rightarrow f(t,\cdot) \enskip \mbox{weekly in } \enskip L^1(\RR^3).
\]
Since $\lambda \log(1+ \lambda)$ is convex, for $0<t\le T$, we have 
\[
\int f(t,v) \log (1+f(t,v)) dv \le \lim \inf \int f_m(t,v) \log (1+f_m(t,v))dv  \le C_T.
\]
And this completes the proof of the proposition.
\end{proof}

\bigskip

\noindent
{\bf Acknowledgements:}
The authors would like to thank Professor Villani for the stimulating
discussion on this topic.


\end{document}